\documentclass[12pt]{article}   	
\usepackage{amsmath,amssymb}

\newtheorem{thm}{Theorem}[section]
\newtheorem{exm}[thm]{Example}

\newtheorem{lem}[thm]{Lemma}
\newtheorem{definition}[thm]{Definition}

\newtheorem{cor}[thm]{Corollary}

\newcommand{\BlackBox}{\rule{1.5ex}{1.5ex}}  
\newenvironment{proof}{\par\noindent{\bf Proof\
}}{\hfill\BlackBox\\[2mm]}

\newcommand{\Real}{\mathit{R}}
\newcommand{\RPlus}{\textit{R}_{+}}

\newcommand{\Zplus}{\textit{Z}_{+}}
\newcommand{\abs}[1]{\left\vert#1\right\vert}

\newcommand{\set}[1]{\left\{#1\right\}}
\newcommand{\tA}{\tilde{\mathcal{\A}}}

\newcommand{\supp}{\texttt{supp}}
\newcommand{\al}{\alpha}
\newcommand{\allone}{\ell}
\newcommand{\be}{\beta}

\newcommand{\si}{\sigma}
\newcommand{\A}{\mathcal{A}}
\newcommand{\B}{\mathcal{B}}
\newcommand{\C}{\mathcal{C}}
\newcommand{\D}{\mathcal{D}}
\newcommand{\F}{\mathcal{F}}

\newcommand{\J}{\mathcal{J}}

\newcommand{\R}{\textbf{R}}
\newcommand{\bcp}{\set{0,1}\!-\!cp}
\newcommand{\cS}{\mathcal{S}}

\newcommand{\cP}{\mathcal{P}}

\newcommand{\cm}{\mathcal{\textbf{m}}}

\newcommand{\T}{\mathcal{T}}
\newcommand{\bu}{\textbf{u}}
\newcommand{\bv}{\textbf{v}}

\newcommand{\bx}{\textbf{x}}

\newcommand{\bE}{\mathbb{E}}

\newcommand{\SF}{\mathbb{SF}}

\newcommand{\mnrt}{$m$th order $n$-dimensional real tensor }
\newcommand{\mnrts}{$m$th order $n$-dimensional real tensors }
\newcommand{\mnst}{$m$th order $n$-dimensional symmetric tensor }
\newcommand{\mnsts}{$m$th order $n$-dimensional symmetric tensors }

\newcommand{\beq}{\begin{equation}}
\newcommand{\eeq}{\end{equation}}
\newcommand{\bey}{\begin{eqnarray}}
\newcommand{\eey}{\end{eqnarray}}
\newcommand{\beyy}{\begin{eqnarray*}}
\newcommand{\eeyy}{\end{eqnarray*}}

\title{$\set{0,1}$ Completely Positive Tensors and Multi-Hypergraphs}
\author{Changqing Xu\thanks{School of Mathematics and Physics, Suzhou University of Science and Technology, Suzhou, P.R. China. Email: cqxurichard@mail.usts.edu.cn.},
Ziyan Luo\thanks{State Key Laboratory of Rail Traffic Control and Safety, Beijing Jiaotong University, Beijing 100044, P.R. China. Email: starkeynature@hotmail.com}, and
Liqun Qi\thanks{Department of Applied Mathematics, The Hong Kong Polytechnic University, Hung Hum, Hong Kong. Email: liqun.qi@polyu.edu.hk}}

  \makeatletter
      \def\@setcopyright{}
      \def\serieslogo@{}
      \makeatother
 \date{\today}
\begin{document}
\maketitle

\begin{abstract}
Completely positive graphs have been employed to associate with completely positive matrices for characterizing the intrinsic zero patterns. As tensors have been widely recognized as a higher-order extension of matrices, the multi-hypergraph, regarded as a generalization of graphs, is then introduced to associate with tensors for the study of complete positivity. To describe the dependence of the corresponding zero pattern for a special type of completely positive tensors--the $\{0,1\}$ completely positive tensors, the completely positive multi-hypergraph is defined. By characterizing properties of the associated multi-hypergraph, we provide necessary and sufficient conditions for any $(0,1)$ associated tensor to be $\{0,1\}$ completely positive. Furthermore, a necessary and sufficient condition for a uniform multi-hypergraph to be a completely positive multi-hypergraph is proposed as well.

\end{abstract}

\noindent \textbf{keywords:} \  Completely positive tensor; $\{0,1\}$ completely positive tensor; multi-hypergraph ;  $(0,1)$ tensor.


\section{Introduction}
\setcounter{equation}{0}

Completely positive matrices (cp matrices)\cite{BM2003,Xu04}, as a special type of nonnegative matrices, have wide applications in combinatorial theory including the study of block designs \cite{Hall1967}, and in optimization
especially in creating convex formulations of NP-hard problems, such as the quadratic assignment problem in combinatorial optimization and the polynomial optimization problems \cite{AKK2013,AKKT2014-1,AKKT2014-2,KKT2013,ZF2014}. The verification of cp matrices is generally NP-hard unless for small scale matrices. For example, all $n\times n$ nonnegative symmetric positive semidefinite matrices (usually called the doubly nonnegative (dnn) matrices) are cp-matrices whenever $n\leq 4$ \cite{BP79}. For general case, it is obvious that cp matrices are dnn, but not always true conversely \cite{BAD2009,DA2013,SZ1993}. It depends on some inherited zero pattern which cp matrices possess. To describe this dependence, the tool of graphs was employed and the completely positive graph (cp graph) was introduced which has all its nonnegative associated matrices being cp. Among all those properties on cp-graphs, one of the most important and well-known is a graph to be completely positive if and only if it does not have an odd cycle of length greater than 4 \cite{KB93}. This gives us a very efficient way to verify cp-matrices in terms of cp graphs.

\vskip 2mm

Recently, the concept of cp matrix has been extended to the higher order cp tensor, which admits its definition in a pretty natural way as initiated by Qi et al. in \cite{Qi13}. Analog to the matrix case, the cp tensors were employed to reformulate polynomial optimization problems \cite{PVZ2015}. Numerical optimization for the best fit of completely positive tensors with given length of decomposition was formulated as a nonnegative constrained least-squares problem in Kolda's paper \cite{K2015}. For the verification of cp tensors, an efficient approach in terms of truncated moment sequences for checking completely positive tensors was proposed and an optimization algorithm based on semidefinite relaxation for completely positive tensor decomposition was established by Fan and Zhou in \cite{FZ2014}. This approach was later accelerated with some preprocessing steps by Luo and Qi in \cite{LQ2015}. Some structured and geometrical properties on general cp tensors were also discussed in \cite{LQ2015,QXX14}.

\vskip 2mm

Inspired by the technique of using cp graphs for the characterization of cp matrices, we employ the multi-hypergraph as a tool to describe the inherited zero pattern for cp tensors, which can further assist with the verification of cp-tensors. Multi-hypergraphs appeared in the literature at least in 1988 or even earlier.  Here we use the definitions in \cite{PZ13}.
Due to complexity of cp-tensors for general higher order cases, we will focus on a special type of cp-tensors called the $\{0,1\}$ cp tensor, which is exactly a higher order extension of the $\{0,1\}$ cp matrix that has been well studied in \cite{BX05,BX07} motivated by the applications in many fields such as the pattern recognition \cite{LD06}. In order to verify $\{0,1\}$ cp tensors, we first build up the correspondence between multi-hypergraphs and symmetric tensors which are called the associated tensors. The $(0,1)$ associated tensor is also defined which is uniquely determined by the corresponding multi-hypergraph. Based on the aforementioned one-to-one relationship, we establish the necessary and sufficient conditions for a $(0,1)$ associated tensor to be  $\{0,1\}$ cp  in terms of some structure property possessed by the corresponding uniform multi-hypergraph. For general $\{0,1\}$ cp tensors which are not necessarily to be $(0,1)$ tensors, the cp multi-hypergraph is introduced and the necessary and sufficient condition of this type of multi-hypergraph is proposed. All of these can not only be served for verification for cp tensors, but also build up a bridge between tensor analysis and multi-hypergraph theory.

\vskip 2mm

The rest of the paper is organized as follows. In Section 2, we introduce the associated tensors for multi-hypergraphs. Some related concepts and properties are also presented. In Section 3, the $\{0,1\}$ cp tensors is introduced and the equivalence conditions for $(0,1)$ associated tensors of multi-hypergraphs to be $\{0,1\}$ cp are proposed. In Section 4, the cp multi-hypergraph is defined in terms of $\{0,1\}$ cp tensors, and the necessary and sufficient condition of cp multi-hypergraphs is established.

\vskip 2mm

Throughout the paper we denote by $[n]$ the set $\set{1,2,\ldots, n}$ for a positive integer $n$, $\abs{S}$ for the cardinality of set (or multiset) $S$, and $\Zplus^{n}$ for the set of
 nonnegative integral vectors of dimension $n$.  Denote by $\T_{m,n}$ the set of all \mnrts, and $\cS_{m,n}$ the set of all \mnsts. Denote $R^n$ the real $n$-dimensional Euclidean space and $R^n_+$ the set of all
nonnegative vectors in $R^n$. Let $\F:=\set{0,1}$ and denote by $\F_{m,n}$ the
 set of all \mnrts whose elements are either 1 or 0, and by $\SF_{m,n}$ the set of  all symmetric tensors in $\F_{m,n}$.  As convention we denote
\[ S(m,n):=\set{\tau=(i_1,i_{2},\ldots,i_m):   i_1,i_{2},\ldots ,i_m\in [n]}\]
for the index set of  an element of an $m$th order tensor.  
For a vector $\bx\in \Real^{n}$, we use $\supp(\bx)$ to denote the \emph{support }of $\bx$, i.e., the index set of the nonzero coordinates of $\bx$.\\

\section{The Multi-Hypergraph and Its Associated Tensor}
\setcounter{equation}{0}
In this section, the multi-hypergraph and its associated tensors are recalled and introduced, and some related concepts and properties are presented. \vskip 2mm

\begin{definition}[Definition 7, \cite{PZ13}] Let $V=\set{v_{1},v_{2},\ldots, v_{n}}$. A \emph{multi-hypergraph} $\cP$ is a pair $(V, \bE)$, where $\bE=\set{E_{1},\ldots, E_{N}}$ a set of multisets of $V$. The elements of $V$ are called the \emph{vertices} and the elements of $\bE$ are called the \emph{edges}. Moreover, a multi-hypergraph $\cP$ is called an \emph{$n\times N$ multi-hypergraph} if $\abs{V}=n, \abs{\bE}=N$.
\end{definition}

\begin{definition}[Definition 8, \cite{PZ13}]  A multi-hypergraph $\cP=(V,\bE)$ is called \emph{$m$-uniform} ($m\geq 2$) if for all $E \in \bE$, the cardinal number of the multiset of $E$ is $m$ (including repeated memberships).
\end{definition}

In this paper, we are interested in $m$-uniform multi-hypergraph.  For simplicity, let $V=[n]$. The associated tensor of an $m$-uniform multi-hypergraph is defined as follows. Unless otherwise stated, we will use $\set{i_1,\cdots, i_m}$ to denote the multiset including repeated memberships throughout the paper.

\begin{definition}\label{asso} Let $V=[n]$. A tensor $\A=(a_{i_1\cdots i_m})\in \cS_{m,n}$ is said to be an associated tensor with the $m$-uniform multi-hypergraph $\cP=(V,\bE)$  if for all $(i_1,\cdots,i_m) \in S(m,n)$, $a_{i_1\cdots i_m}\neq 0$ when the multiset $\set{i_1,  \cdots, i_m}$ forms an edge in $\bE$, and $a_{i_1\cdots i_m}=0$ otherwise.
\end{definition}

\vskip 2mm

Let $\al\in \bE$. We use $B(\al)$ to denote the set consisting of all distinct elements of $\al$ and call it the \emph{base} of $\al$. Apparently, any hypergraph (see \cite{Be89} for details) is a multi-hypergraph with $\al=B(\al)$ for each edge $\al\in \bE$. Since the repetition is allowed in each edge for a general multi-hypergraph (i.e., $\B(\al)\subseteq \al$), some partial order can be induced for edges in terms of their bases.
%
Let $\al, \be\in \bE$. $\al$ is said to be \emph{majorized} (\emph{strictly majorized}) by an edge $\be$,  denoted as $\al\preceq \be$ ($\al\prec \be$),  if
$B(\al)\subseteq  B(\be)$ ($B(\al)\subset B(\be)$).  $\al$ and $\be$ are said to be \emph{similar}, denoted as $\al\sim \be$,  if both $\al\preceq \be$ and $\al\succeq \be$ hold.
Similar edges have a common base. 
The majorization defines a partial order on $\bE$  and gives a clustering of edges in $\bE$,
say $\D_{1},\D_{2}, \cdots, \D_{r}$ (possibly with some overlappings).  By Zorn's Lemma, there exists at least one \emph{maximal} (\emph{minimal}) element in each
$\D_{i}$, denoted as $\ell_{i}$ ($\cm_{i}$, respectively) which satisfies
\[  B(\al) \subseteq B(\ell_{i}) ~\left( B(\al) \supseteq B(\cm_{i}),\text{~respectively}\right), \quad  \forall \al\in \D_{i}  \]
%
for each $i=1,2,\ldots,r$.  For an $m$-uniform $n\times N$ multi-hypergraph $\cP=(V,\bE)$, it is obvious that $1\le \abs{\cm_{i}}\le \abs{\ell_{i}}\le m$. Denote
\[ rk(\cP):=\max\limits_{1\leq i\leq r}\{\abs{B(\ell_{i})}: \ell_{i} \text{ is a maximal edge of}~\D_{i}\},\]
and
\[ ck(\cP):=\min\limits_{1\leq i\leq r}\{\abs{B(\cm_{i})}: \cm_{i} \text{ is a minimal edge of}~\D_{i}\}\]
$rk(\cP)$($ck(\cP)$, respectively) is called the \emph{rank} (\emph{co-rank}) of $\cP$.  For an $m$-uniform hypergraph $\cP$ , we have $rk(\cP)=ck(\cP)=m$.
\vskip 2mm

\indent A multi-hypergraph $\cP$ may have several maximal (minimal) edges with different bases. But its rank (co-rank) shall be a unique number by definition.
For any $\al\in S(m,n)$, we denote by $M(\al)$ the multiset generated by $\al$ and define the \emph{complete $m$-multiset determined by $\al$} as
\[ \D_{\al}:=\set{\eta\in S(m,n): M(\eta)\preceq M(\al)}. \]

\begin{lem}\label{lem0}
Let $\abs{B(\al)}=r$ where $\al\in S(m,n)$.  Then
\beq\label{eq:le00}
\abs{\D_{\al}} = r^{m}.
\eeq
\end{lem}

\begin{proof}
Let $\al\in S(m,n)$, and $\abs{B(\al)}=r$.  We may assume w.l.g. that $B(\al)=\set{s_{1}, s_{2},\ldots, s_{r}},  1\le s_{1}< s_{2}<\ldots < s_{r}\le n$. For each
$\eta:=(i_{1},i_{2},\ldots, i_{m})\in \D_{\al}$, its coordinate $i_{k}$ can be any number chosen from $B(\al)$ for each $k\in [m]$, and thus there are $r^{m}$ choices, which leads to the desired assertion.
\end{proof}

\indent The following example is presented for the illustration of the above concepts.
\begin{exm}\label{exm01} Let $\cP=(V, \bE)$ be a $3$-uniform multi-hypergraph with its associated tensor  $\A=(A_{ijk})\in \cS_{3,3}$  whose nonzero elements are listed as below:
\[ A_{112}=A_{122}=A_{133}=A_{113}=A_{223}=A_{111}=A_{222}=A_{333}=1 \]
Three complete 3-multisets of  $\bE$ given by the majorization are
\[ \D_{1} =\set{\set{1,1,2},\set{1,2,1},\set{2,1,1}, \set{1,2,2},\set{2,1,2},\set{2,2,1},\set{1,1,1},\set{2,2,2}}, \]
\[ \D_{2} =\set{\set{1,1,3},\set{1,3,1},\set{3,1,1}, \set{1,3,3},\set{3,1,3},\set{3,3,1},\set{1,1,1},\set{3,3,3}}, \]
\[ \D_{3} =\set{\set{2,2,3},\set{2,3,2},\set{3,2,2}, \set{2,3,3},\set{3,2,3},\set{3,3,2},\set{2,2,2},\set{3,3,3}}. \]
There are six maximal edges in each $\D_{i}$. In fact, all the edges but the three minimal edges $\set{1,1,1},\set{2,2,2},\set{3,3,3}$ are the maximal edges.
So $rk(\cP)=2, ck(\cP)=1$.  Note that $\set{\D_{1},\D_{2},\D_{3}}$ does not form a partition of  $\bE\subset S(3,3)$ since
\[ \set{1,1,1}\in \D_{1}\cap \D_{2}\neq \emptyset. \]
\end{exm}

\section{$\{0,1\}$ cp Tensors and $(0,1)$ Associated Tensors}
\setcounter{equation}{0}
In this section, we will discuss the condition for $(0,1)$ associated tensors of uniform multi-hypergraphs to be $\{0,1\}$ cp tensors. Before stating the main theorem, the involved concepts are introduced as a start. \vskip 2mm

\begin{definition} An \mnst $\A$ is called a \emph{completely positive tensor}, or \emph{cp tensor} for short,  if $\A$ can be decomposed as
\beq\label{eq:cp01}
\A = \sum\limits_{j=1}^{q} \bu_{j}^{m}, \quad   \bu_{j}\in \RPlus^{n}, \forall j\in [q]
\eeq
The smallest number $q$ satisfying (\ref{eq:cp01}) is called the \emph{cp-rank} of $\A$.  Moreover, $\A$ is called to be \emph{$\bcp$} if  $\bu_{j}\in\F^{n}$ for each $j\in [q]$ in (\ref{eq:cp01}).
\end{definition}

Note that a $\bcp$ tensor may not be a $(0,1)$ tensor, and a cp-tensor with all entries in $\F$ is not necessarily $\bcp$. \\

By Definition \ref{asso}, it is obvious that for a given $m$-uniform multi-hypergraph $\cP$, its associated tensors are infinitely many since we can put any nonzero scalars as entries in those positions corresponding to edges of the multi-hypergraph. If we further restrict the associated tensor $\A$ to be in $\SF_{m,n}$, then the correspondence turns out to be one-to-one, i.e.,
\beq \label{01} \{i_{1}, i_{2},\ldots, i_{m}\}\in \bE  \Longleftrightarrow a_{i_{1}i_{2}\ldots i_{m}}=1. \eeq
For any tensor $\A=(a_{\si})\in \T_{m,n}$, a \emph{tensor pattern} $\tA=(\tilde{a}_{\si})\in \F_{m,n}$ is defined in the way that for any $\si\in S(m,n)$, $\tilde{a}_{\si}=1$ if $a_{\si}\neq 0$ and $\tilde{a}_{\si}=0$ otherwise. Apparently, all associated tensors for an $m$-uniform multi-hypergraph share the same tensor pattern which is exactly the corresponding $(0,1)$ associated tensor. The pattern of a tensor $\A$ reflects the distribution of zero (nonzero) elements of  $\A$ and thus can be used to characterize its  spectral property e.g. \cite{CPZ08,CPZ09,CPZ11,FGH09,Pe10,YY10,YY11} and combinatorial properties such as the irreducibility.

\begin{definition}[Definition 2.1, \cite{CPZ08}] An \mnrt $\A=(a_{i_1\cdots i_m})\in \T_{m,n}$ is called \emph{reducible}   if there is a nonempty proper subset $I\subset [n]$ such that
\beq\label{eq:defreducible}
a_{i_{1}\ldots i_{m}} =0, ~~\forall i_{1}\in I,~\forall i_{2}, \ldots,i_{m} \notin I.
\eeq
$\A$ is called \emph{irreducible} if it is not reducible.
\end{definition}

Recall that a \emph{slice} of  tensor $\A\in \T_{m,n}$ is defined as a sub-tensor of order $m-1$ obtained from $\A$ with some index fixed.  A \emph{zero slice}, or a \emph{trivial} slice, is a slice
whose elements are all zeros. Given a nonempty subset $I:=\set{s_{1}, s_{2},\ldots, s_{r}}$ of $[n]$, a \emph{principal subtensor} of $\A$ determined by $I$, is defined as the $m$th order $r$ dimensional tensor $\B=(A_{i_{1}i_{2}\ldots i_{m}})$ where each $i_{k}$ is constrained in $I$. A \emph{zero block} is a principal subtensor whose entries are all zero. Obviously, an irreducible tensor has no zero slice nor any zero block.

\vskip 2mm

Reducibility is a pattern property for tensors. By employing the permutational similarity property, we can decompose any $(0,1)$ reducible tensor into a direct sum of a finite number of low dimensional irreducible tensors and a zero tensor in the permutational similar sense. Before stating this result, some related concepts are recalled here. Let $\A, \B\in \T_{m,n}$. We say that $\A$ is \emph{permutational similar} to $\B$, denoted as $\A\sim_{p} \B$,  if there exists a permutation matrix $P\in \Real^{n\times n}$
such that
\[ \B = \A\times_{1} P\times_{2}  P\times_{3} \cdots \times_{m} P, \]
where $\tA:=\A\times_{k} P=(\tilde{a}_{i_{1}\ldots i_{m}})\in \T_{m,n}$ is defined as
\[ \tilde{a}_{i_{1}\ldots i_{k-1}i_{k}i_{k+1}\ldots i_{m}} = \sum\limits_{j=1}^{n} a_{i_{1}\ldots i_{k-1}j i_{k+1}\ldots i_{m}}p_{i_{k}j} \]

\vskip 2mm

Utilizing the permutational similarity of tensors, we can build up some identical relation among their corresponding multi-hypergraphs. Let $\cP_{1}=(V_{1}, \bE_{1})$ and $\cP_{2}=(V_{2}, \bE_{2})$ be two given $m$-uniform multi-hypergraphs with their $(0,1)$ associated tensors $\A$ and $\B$ respectively.
Then $\A\sim_{p} \B$ if and only if there exists a bijection $\phi$ from $V_{1}$ to $V_{2}$ such that
\[  \{i_{1},i_{2},\ldots,i_{m}\}\in \bE_{1}  \mapsto \{\phi(i_{1}),\phi(i_{2}),\ldots,\phi(i_{m})\}\in \bE_{2} \]
that is, $\cP(\B)$ is the multi-hypergraph obtained from $\cP(\A)$ by the reordering of its vertices, and thus they are identical in this sense. \\

Let $\A_{i}=(a_{\si}^{(i)})\in \T_{m,n_{i}}, i=1,2$ and $n_{1}+n_{2}=n$.  The \emph{direct sum} of $\A_{1}$ and $\A_{2}$, denoted by
\[ \A = \A_{1}\oplus \A_{2} =(a_{i_{1}\ldots i_{m}}), \]
is defined by
\[ a_{i_{1}\ldots i_{m}}=\begin{cases} a_{i_{1}\ldots i_{m}}^{(1)}  & \text{if} i_{1},\ldots, i_{m}\in [n_{1}], \\
                                                                            a_{i_{1}\ldots i_{m}}^{(2)}  & \text{if} i_{1},\ldots, i_{m}\in n_{1}+[n_{2}], \\
      0& \text{otherwise}.
\end{cases}  \]
Here $a+S$ is defined as the translation of set $S$,  i.e., $a+S=\set{a+s: s\in S}$.\\

\vskip 2mm

Now we are in a position to describe the decomposition for tensors in the sense of permutation similarity.

\begin{lem}\label{lem1}
Let $\A\in \SF_{m,n}$, where $m\ge 2, n\ge 1$. Then
\beq\label{eq:irredufact}
\A\sim_{p} \A_{1}\oplus \A_{2}\oplus \ldots \oplus \A_{r}\oplus \mathcal{O}_{r+1}
\eeq
where $\A_{i}\in \SF_{m,n_{i}}$ is irreducible, $\mathcal{O}_{r+1}$ is a zero tensor of order $m$ and dimension $n_{r+1}$, and $n_{1}+\ldots + n_{r+1}=n$.
\end{lem}

\begin{proof} The result is trivial if $\A$ is irreducible tensor. Now we assume that $\A\in \SF_{m,n}$ is a reducible tensor. We will use induction to prove the desired statement. For $n=1$, the reducibility implies that $\A=0$. The statement holds by setting $r=0$. Assume that for all $k$ satifying $1\leq l\leq n$ with $n \geq 1$, the statement holds. They for the case of $l+1$, there exists a nonempty subset $I$ of $[l+1]$ such that
\beq\label{eq:leprf1}
 A_{i_{1}i_{2}\ldots i_{m}}=0, \forall i_{1}\in I, i_{2},\ldots, i_{m}\notin I
\eeq
Let $\cP=(V, \bE)$ be the multi-hypergraph with $\A$ as an associated tensor, and we assume w.l.g. that
\[ I:=\set{k_{1}, k_{2},\ldots, k_{r}}, 1\le k_{1} <k_2 <\cdots < k_r\le l+1.\]
Then we let $\phi: [l+1]\to [l+1]$ be an one-to-one correspondence such that
\[ \phi(k_{i})=i, \quad  \forall\  i=1,2,\ldots, r. \]
and $\phi$ maps $[l+1]\backslash I$ to $[l+1]\backslash [r]$.   $\phi$ can be regarded as a permutation on $[l+1]$, and so there is a permutation matrix $P$ corresponding to $\phi$.
Actually if we define $P=(p_{ij})\in \F^{(l+1)\times (l+1)}$ by
\[ p_{ij}=1 \quad \texttt{iff}\  j=\phi(i)  \]
for each $i\in [l+1]$. It follows readily that
\beq\label{eq:directsum}
\tA:= \A\times_{1}P\times_{2}P\times_{3} \ldots \times_{m}P  = \A_{11}\oplus \A_{22}
\eeq
where $\A_{11}\in \SF_{m,r}, \A_{22}\in \SF_{m,l+1-r}$. Note that $r$, $l+1-r\leq n$,  the desired decomposition can be proved by the induction.
\end{proof}

\vskip 2mm 
Lemma \ref{lem1} shows that a tensor $\A\in \SF_{m,n}$ can always be decomposed into the direct sum of irreducible tensors, possibly with a zero block.  The following lemma is dedicated to the necessary and sufficient conditions of $\bcp$ property for irreducible $(0,1)$ tensors.

\vskip 2mm

\begin{lem}\label{lem2}
Let $m\ge 2, n\ge 1$ be two positive integers, and $\A\in \SF_{m,n}$ be irreducible.  Then the following statements are equivalent:
(i) $\A$ is $\bcp$;

(ii) $\A=\J$ is the all-$1$ tensor;

(iii) the multi-hypergraph $\cP$ with associated tensor $\A$ is a complete block.
\end{lem}

\begin{proof}
If $\A=\J$, then surely $\A$ is $\bcp$ since $\A=\allone^{m}$ with $\allone=(1,1,\ldots,1)^{\top}$. \\
\indent  Conversely, we let $\A\in \SF_{m,n}$ be a $\bcp$ tensor. Then $\A$ has a decomposition (\ref{eq:cp01}) with
\[ \bu_{j}=(u_{1j},u_{2j},\ldots, u_{nj})^{\top}\in \F^{n}. \]
Then we have
\[a_{i_{1}i_{2}\ldots i_{m}} = \sum\limits_{j=1}^{q} u_{i_{1}j}u_{i_{2}j}\ldots u_{i_{m}j}, ~~ \forall (i_{1},i_{2},\ldots,i_{m})\in S(m,n).  \]
We will first show that $q=1$ in decomposition (\ref{eq:cp01}).  Suppose that $q>1$.  If there exist a pair of positive integers $(s,t): 1\le s < t \le q$ such that
\[  k\in \supp(\bu_{s})\cap \supp(\bu_{t})  \]
for some $k\in [n]$, then $u_{ks}=u_{kt}=1$. Hence we have
\beyy
A_{kk\ldots k} & = & \sum\limits_{j=1}^{q} u_{kj}u_{kj}\ldots u_{kj} \\
                          & = & \sum\limits_{j=1}^{q} u_{kj}^{m}\\
                          &\ge & u_{ks}^{m} + u_{kt}^{m}=2
\eeyy
a contradiction to the assumption that $\A$ is a (0,1) tensor.  Thus we have
\beq\label{eq: disjointsupp}
\supp(\bu_{i})\cap \supp(\bu_{j}) = \emptyset, \forall 1\le i < j\le q
\eeq
Now we define
\[ \D_{i}=\set{\si\in \bE: B(\si)\subseteq \supp(\bu_{i})}, \forall i=1,2,\ldots,q \]
Then we get $\set{\D_{1},\D_{2}, \ldots,\D_{q}}$ each a subset of $\bE$,  and
\[ \D_{i}\cap \D_{j}=\emptyset, \quad   \forall 1\le i < j\le q   \]
Denote $V_{i}=V(\D_{i})$ and $\cP_{i}:=(V_{i}, \D_{i})$ for $i=1,2,\ldots,q$.  Then
\[ \cP = \cP_{1}\cup \cP_{2}\cup \ldots \cup \cP_{q} \]
where $\cP=(V, \bE)$ is the multi-hypergraph associated with $\A$. It turns that $\A\sim_{p} \A_{1}\oplus \ldots \oplus\A_{q}$ where $\A_{i}$ is the adjacency tensor of $\cP_{i}$, a
contradiction to the hypothesis that $\A$ is irreducible.  Hence $q=1$,  and thus there exists a vector $\bu=(u_{1},\ldots,u_{n})^{\top}\in \F^{n}$ such that $\A=\bu^{m}$. \\
\indent To prove that $\A=\J=\allone^{m}$, we need only to show that  $\supp(\bu)=[n]$. In fact, if $\supp(\bu)$ is a proper subset of $[n]$, then by setting
$I=[n]\backslash \supp(\bu)$, we show that $\A$ is reducible by definition, which is a contradiction to the hypothesis. Thus $\supp(\bu)=[n]$ and  $\A=\J$. Thus the equivalence between (i) and (ii) is obtained.
\indent The remaining part of the lemma is immediate by definition.
\end{proof}

 \indent From Lemma \ref{lem2} and its proof, we can get the following equivalences for $\bcp$ tensors.
\begin{thm}\label{thm1}
Let $m\ge 2, n\ge 1$ be two positive integers. Suppose that $\A\in \SF_{m,n}$ have no zero blocks and is associated with multi-hypergraph $\cP=(V,\bE)$.  Then the following are equivalent:
\begin{description}
\item[(1)]  $\A$ is $\bcp$\  tensor.
\item[(2)]  $\cP$ can be decomposed as the union of some complete blocks $\cP_{i}$ of size $n_{i}$ where $n_{1}+\ldots +n_{q}=n$.
\item[(3)]  $\A$ can be written in form (\ref{eq:cp01}) and with $\bu_j\in\F^n$ satisfying  $U^{T}U=diag(n_{1},\ldots,n_{q})$ where $U=[\bu_{1},\ldots,\bu_{q}]$.
\end{description}
\end{thm}

\begin{proof} To prove $(1) \Leftrightarrow (2)$, we first let $\A\in \SF_{m,n}$ be a $\bcp$ tensor.  Then by Lemma \ref{lem1} $\A$ can be written in form (\ref{eq:irredufact}) where each $\A_{i}$
is an irreducible $\bcp$ tensor of $m$th order $n_{i}$-dimension (no zero block there since $\A$ has no zero block).  By Lemma \ref{lem2}, $\A_{i}$ is associated
with a multi-hypergraph $\cP_{i}=(V_{i}, \bE_{i})$ where $\abs{V_{i}}=n_{i}$ for $i=1,2,\ldots, q$, $n_{1}+n_{2}+\ldots +n_{q}=n$.  For each $i\in [q]$, by Lemma \ref{lem2},  $\cP_{i}$ is
the complete block of  dimension $n_{i}$ (since $\A_{i}$ is irreducible and  $\bcp$ ). Thus $(1)\Rightarrow (2)$ is proved.  The proof of  $(2)\Rightarrow (1)$ is immediate if we
note that the decomposition (\ref{eq:cp01}) holds by take $\supp(\bu_{i}) =V_{i}$ for $i=1,2,\ldots, q$. \\
\indent Now we show $(1)\Leftrightarrow (3)$.  First we assume that $\A\in \SF_{m,n}$ is $\bcp$.  Then from the proof of Lemma \ref{lem2} there exist some vectors
$\bu_{j}\in \F^{n}$ such that (\ref{eq:cp01})  holds, and
\beq\label{eq: disjointsupp}
\supp(\bu_{i})\cap \supp(\bu_{j}) = \emptyset, \forall 1\le i < j\le q
\eeq
It follows that $U^{T}U=diag(n_{1},\ldots, n_{q})$ for $U=[\bu_{1},\ldots,\bu_{q}]$, where $n_{i}$ is the positive integer described above. Thus $(1)\Rightarrow (3)$ is proved.
The other direction can be proved by reversing the above arguments.
\end{proof}

\section{cp Multi-Hypergraphs}
\setcounter{equation}{0}
We define a multi-hypergraph $\cP=(V,\bE)$ to be a \emph{cp} multi-hypergraph if $\cP$ is associated with a $\bcp$ tensor $\A$. Note that $\A$ is not necessarily a (0,1) tensor. For example,
the following $3\times 3\times 3$ symmetric tensor is a $\bcp$, but not a (0,1) tensor.\\
\begin{exm}\label{exm03}
Let $\A=(a_{ijk})\in \cS_{3,3}$ be a symmetric tensor defined as:\\
\[ \A(:, :, 1) =\left[\begin{array}{ccc} 2&1&1\\ 1&1&0\\ 1&0&1\end{array} \right] \]
\[ \A(:, :, 2) =\left[\begin{array}{ccc} 1&1&0\\ 1&2&1\\ 0&1&1\end{array} \right] \]
\[ \A(:, :, 3) =\left[\begin{array}{ccc} 1&0&1\\ 0&1&1\\ 1&1&2\end{array} \right] \]
We show that $\A$ is a $\bcp$ tensor. In fact, if we let
\[ U =\left[\begin{array}{ccc} 1  &  1  &  0\\  1  &  0  &  1\\ 0  &  1  &  1\end{array}  \right]  \]
Then we can verify by simple computation that
\[ \A =\bu_{1}^{3} +\bu_{2}^{3}+\bu_{3}^{3} \]
where $\bu_{1},\bu_{2},\bu_{3}\in \F^{3}$  are respectively the first, second and the third column of the (0,1) matrix $U$.  Thus $\A$ is $\bcp$ by definition.  But
$\A$ is not a (0,1) tensor since $ a_{111}=a_{222}=a_{333}=2$.\\
\indent  Denote $\cP$ as the associated multi-hypergraph of $\A$.  There are three distinct classes of edges of $\cP$ according to majorization:
$\D_{\al_{1}},\D_{\al_{2}} ,\D_{\al_{3}}$ where
\[ \al_{1}=\set{v_{1},v_{1},v_{2}},\quad \al_{2}=\set{v_{1},v_{1},v_{3}},\quad \al_{1}=\set{v_{2},v_{2},v_{3}} \]
each pair $(\D_{i}, \D_{j})$ has a nonempty intersection. Note that $\A$ is irreducible. Thus the condition $\A\in \SF_{m,n}$ in Lemma \ref{lem2} cannot be removed.
\end{exm}

The following property is introduced for cp multi-hypergraphs.

\begin{definition}\label{R} A multi-hypergraph $\cP=(V, \bE)$ is said to possess Property $\R$ if  $\D_{\al} \subseteq \bE$ for any $\al\in \bE$.
\end{definition}

The aforementioned property is closely related to the zero-entry dominance property for tensors described formally by Luo and Qi in \cite{LQ2015}.

\begin{definition}[Definition 4.1, \cite{LQ2015}]
An $m$th order $n$-dimensional tensor $\A=(a_{i_1\cdots i_m})$ is said to possess the \emph{zero-entry dominance property} if for any $(i_1,\cdots, i_m)\in S(m,n)$, $a_{i_1\cdots i_m}=0$ implies that $a_{j_1\cdots j_m}=0$ for all $\{j_1,\cdots, j_m\}$ satisfying $B(\{j_1,\cdots, j_m\}) \supseteq B(\{i_1,\cdots, i_m\})$.
\end{definition}

By direct verification, we can get the following equivalence.

\begin{lem} A multi-hypergraph $\cP=(V, \bE)$ has the Property $\R$ if and only if its $(0,1)$ associated tensor has the zero-entry dominance property.
\end{lem}

It is known from \cite{LQ2015} that any completely positive tensor possesses the zero-entry dominance property.
It is worth pointing out that the zero-entry dominance property is only a necessary condition for $(0,1)$ associated tensor to be $\{0,1\}$ cp tensors, but far away from sufficient, even for the matrix case. For example, $A=[1~1~0;1~1~1; 0~1~1]$ is a $(0,1)$ matrix and satisfies the zero-entry dominance property, but it is not $\{0,1\}$ cp since it is not positive semidefinite. Nevertheless, by invoking the above equivalence between Property $\R$ and the zero-entry dominance property, together with the definition of cp multi-hypergraphs, we can obtain that the Property $\R$ is exactly a necessary and sufficient condition for a multi-hypergraph to be cp. \vskip 2mm


\begin{thm}\label{thm2}
An $m$-uniform multi-hypergraph is a cp multi-hypergraph if and only if it possesses Property $\R$.
\end{thm}

\begin{proof} To get the necessity, by invoking Theorem \ref{thm1}, a cp multi-hypergraph can be decomposed as the union of  some disjoint complete blocks $\cP_{i}$'s where $\cP_{i}=(V_{i}, \bE_{i})$ is of order $n_{i}$.
Since a complete block has Property $\R$, $\cP$ has Property $\R$. For the sufficiency, that is, given an $m$-uniform multi-hypergraph $\cP=(V, \bE)$ with Property $\R$,  then $\cP$ is cp, i.e., there exists a $\bcp$ tensor
$\A$ associated with $\cP$.  Note that $\A$ need not be a (0,1) tensor by definition. For this purpose, we consider the set of the maximal edges of $\cP$, and we classify them into $\C_{1}, \C_{2},\ldots, \C_{r}$ by the similarity of the edges. We denote
\[\al_{i}\in \C_{i}, \quad  B_{i}=B(\al_{i}), \quad  n_{i}=\abs{B_{i}}, \forall i\in [r]  \]
For $i\in [r]$, denote $\D_{i}=\D_{\al_{i}}$, which is the complete $m$-multisets determined by $\al_{i}\in \C_{i}$. Apparently,
\beq\label{Dunion}
\bE =\bigcup_{i=1}^{r} \D_{i}
\eeq
Denote $\bv_{i}\in \F^{n}, \supp(\bv_{i})=B_{i}$ for each $i\in [r]$, and define
\beq\label{eq: cpfac01}
\A =\sum\limits_{j=1}^{r} \bv_{j}^{m}
\eeq
Then $\A$ is $\bcp$ by definition.  The proof is completed  if we show that $\A$ is associated with $\cP$.  In fact, if  there is an element  $a_{i_{1} \ldots i_{m}}\neq 0$, then from (\ref{eq: cpfac01}) there exists $k\in [r]$ such that
\[ B(\si) \subseteq \supp(\bv_{k})=B_{k}=B(\al_{k}) \]
It follows that $M(\si)\in \bE$ since $\cP$ has Property $\R$.  Conversely, if $A_{\si}=0$ for a $\si\in S(m,n)$, then  $B(\si)\nsubseteq B(\al_{i})$ for each $i\in [r]$. Thus
$\si\notin \D_{i}$ for all $i\in [r]$. Consequently we have $M(\si)\notin \bE$ by (\ref{Dunion}).  The proof is completed.
\end{proof}

\vskip 2mm

By combining Theorem \ref{thm2} and Lemma \ref{lem0}, we obtain
\begin{cor}\label{cor2}
An $m$-uniform $n\times N$ cp multi-hypergraph $\cP=(V,\bE)$ satisfies
\[ N = n_{1}^{m}+\ldots +n_{r}^{m}\]
where $r$ is the number of connected branches of $\cP$ and $n_{i}$ is the dimension of  the $i$th branch.
\end{cor}


\section*{Acknowledgement}
This research was supported by the National Natural Science Foundation of China (11301022,11431002) and the Hong Kong Research Grant Council (No. PolyU 502111, 501212, 501913 and 15302114).  The work was partially done during the first author's  visit at the Hong Kong Polytechnic Uiversity in 2015. The authors would like to thank Prof. M.D. Choi
who shared his valuable views concerning the completely positive factorizations. 

\end{document}